\documentclass{amsart}

\usepackage{amssymb}
\usepackage{amsmath}
\usepackage{amsthm}
\usepackage{amscd}
\usepackage{mathrsfs}
\usepackage{graphicx}
\usepackage{fullpage}
\graphicspath{{Figures/}} 

\usepackage{cancel} 
\usepackage{enumerate} 
\usepackage{hyperref} 
\usepackage[usenames,dvipsnames]{color} 
\usepackage{polynom}
\usepackage{subfig} 

\usepackage{tikz}
\usetikzlibrary{arrows}
\usetikzlibrary{calc}

\definecolor{pgray}{RGB}{75,92,113}
\definecolor{porange}{RGB}{230,148,20}
\definecolor{pblue}{RGB}{104,140,152}

\newcommand{\colorlist}[1]{%
\ifcase#1 
Blue%
\or
ForestGreen%
\or
Fuchsia%
\or
Turquoise%
\or
BurntOrange%
\or
WildStrawberry%
\or
LimeGreen%
\else
Orchid%
\fi
}

\hypersetup{colorlinks=true, linkcolor=pblue}

\theoremstyle{plain} 
\newtheorem{theorem}{Theorem}[section]
\newtheorem{proposition}[theorem]{Proposition}

\theoremstyle{definition} 

\newcounter{suggestcount}
\newcounter{commentlabel}
\newlength{\poincarecommentlift}
\makeatletter
\DeclareRobustCommand{\COMMENT}[1]{\@bsphack%
	\stepcounter{commentlabel}%
	\vbox to0pt{%
		\setlength{\fboxrule}{0.75pt}%
		\setlength{\fboxsep}{0.75pt}%
		\setlength{\poincarecommentlift}{1ex}%
		\addtolength{\poincarecommentlift}{\fboxrule}%
		\addtolength{\poincarecommentlift}{\fboxsep}%
		\vss\color{red}%
		\rlap{\rlap{\vrule height\poincarecommentlift width\fboxrule}\raise \poincarecommentlift%
		\hbox{\fcolorbox{red}{yellow}{%
\normalfont\footnotesize\ttfamily\bfseries\thecommentlabel}}}}%
	\marginpar{\noindent\raggedright%
	\textbf{\color{red}\thecommentlabel}:\thinspace\footnotesize#1}%
}
\title{ Successive extensions  of vector bundles on curves}
\author{montserrat teixidor i bigas}
\address{ Mathematics Department, Tufts University, Medford MA
02155, U.S.A.}

\let\oldtocsection=\tocsection
\let\oldtocsubsection=\tocsubsection
\renewcommand{\tocsection}[2]{\hspace{0em}\oldtocsection{#1}{#2}}
\renewcommand{\tocsubsection}[2]{\hspace{1.1em}\oldtocsubsection{#1}{#2}}

\begin{document} 

\maketitle

\begin{abstract} We show that on a generic curve,  a bundle obtained by successive extensions is stable.
We compute the dimension of the set of such extensions. We use degeneration methods specializing the curve to a chain of elliptic components
\end{abstract}

Take a generic (compact, non-singular) curve $C$ of genus $g$ defined over the complex numbers. 
A vector bundle $E$ on $C$ of rank $r$ and degree $d$ is said to be stable if for every vector subbundle $E_1$ of $E$ of rank $r_1$ and degree $d_1$, 
the inequality $\frac {d_1}{r_1}<\frac dr$ is satisfied.
The moduli space ${\mathcal U}(r,d)(C)$ parametrizes isomorphism classes of stable  vector bundles of rank $r$ and degree $d$ on $C$.
It is a non-singular variety of dimension $r^2(g-1)+1$.

Fix now  $E\in {\mathcal U}(r,d)(C)$ and an integer  $r_1<r$.
Define  $s_{r_1}(E)=r_1d-rmax\{deg E_1\}$ where $E_1$ moves in the set of subbundles of $E$ of rank $r_1$.
As, $E$ is stable  $s_{r_1}(E)> 0$ for all $r_1$. 
On the other hand,  for a generic $E$,   $s_{r_1}(E)$ is the smallest integer greater than or equal to $ r_1(r-r_1)(g-1)$ and congruent with $r_1d$ modulo $r$ (\cite{L} Satz p.448,\cite {LN}).
 Fix then $s$ such that $0<s\le  r_1(r-r_1)(g-1)$. The (proper) subset of the moduli space of vector bundles given as 
\[ {\mathcal U}^s(r,d)(C)= \{ E\in {\mathcal U}(r,d)(C)\text{ such that } s_{r_1}(E)=s \}  \]
 generically coincides with the space of stable bundles $ E_{r,d}$ that fit in an exact sequence
\[ 0\to E_{r_1, d_1}\to E_{r,d}\to \bar{E}_{r-r_1, d-d_1}\to 0\]
 Lange conjectured that a generic choice of the two bundles $E_{r_1, d_1},  \bar{E}_{r-r_1, d-d_1}$ together with a generic choice of the extension would give rise to a stable $E_{r,d}$.
 This would prove that ${\mathcal U}^s(r,d)(C)$ is non-empty and of the expected dimension.
 Partial results were obtained in \cite{LN}, \cite{BR}, \cite{BL} \cite{extlb} \cite{langred} among others and  the conjecture was proved in full generality in \cite{RT}.
 
 Our goal here is to generalize this result to the case of several  successive extensions. We show  the following
 
\begin{theorem}\label{mainth} Let $C$ be a generic curve of genus $g$.
 Fix a rank $r$ and degree $d$.
 Choose a collection of integers $r_1<r_2<\dots<r_k=r$ and degrees $d_1,\dots, d_k=d$ such that 
 \[ \frac {d_1}{r_1}< \frac {d_2}{r_2}<\dots<\frac {d_k}{r_k}\]
 and 
 \[ r_1d_2-r_2d_1\le r_1(r_2-r_1)(g-1),\  r_2d_3-r_3d_2\le r_2(r_3-r_2)(g-1),\dots, r_{k-1}d_k-r_kd_{k-1}\le r_{k-1}(r_k-r_{k-1})(g-1)\]
 Define $ {\mathcal U}(r_1,\dots ,r_k;d_1,\dots, d_k)(C)\subseteq  {\mathcal U}(r,d)(C)$
 as the set  of  stable $E_{r,d}$ obtained after a sequence of extensions 
 \[ 0\to E_{r_1, d_1}\to E_{r_2,d_2}\to \bar{E}_{r_2-r_1, d_2-d_1}\to 0\]
 \[ 0\to E_{r_2, d_2}\to E_{r_3,d_3}\to \bar{E}_{r_3-r_2, d_3-d_2}\to 0\]
\[ \dots\]
 \[ 0\to E_{r_{k-1}, d_{k-1}}\to E_{r,d}\to \bar{E}_{r-r_{k-1}, d-d_{k-1}}\to 0\]
 Then,  $ {\mathcal U}(r_1,\dots ,r_k;d_1,\dots, d_k)(C)$ is non-empty, irreducible and whose codimension  in $ {\mathcal U}(r,d)(C)$ is
 \[  [ r_1(r_2-r_1)+r_2(r_3-r_2)+\dots+r_k(r_k-r_{k-1})](g-1)-[r_1d_2-r_2d_1+r_2d_3-r_3d_2+\dots +r_{k-1}d_k-r_kd_{k-1} ] \]
 \end{theorem}
Following the ideas in \cite{extlb}, \cite{langred}. we will use a degeneration argument.
We first   prove the result for a particular reducible nodal curve and then we extend it to the generic curve.

We became interested in this question while studying families of rational curves in the moduli space of vector bundles on a fixed curve.
The stability of a successive extension is crucial for the construction of families of rational curves in the moduli space.

The result in the case of a single extension has found applications to a variety of other topics, in particular to the study of Brill-Noether theory for vector bundles 
and to the existence of Ulrich bundles on ruled surfaces. 
We expect that this more general result will find similar applications.

\section{The problem on the reducible curve}

Tensoring with  line bundles gives isomorphisms between ${\mathcal U}(r,d)(C)$ and ${\mathcal U}(r,d+tr)(C)$, so there are only, up to isomorphism, 
 at most $r$ non-isomorphic moduli spaces of vector bundles on a curve (in fact, about half of that if you consider also dualization, but this is irrelevant to us now).
 Without loss of generality, we will assume that $0\le d<r$.
 
We will be using bundles on reducible, nodal curves  as limits of vector bundles on non-singular curves. 
More specifically, we will consider chains of elliptic curves defined as follows:
Let $C_1,\dots, C_g$ be $g$ elliptic curves, $P_i,\not= Q_i$ arbitrary points on $C_i$.
Glue $Q_i$ to $P_{i+1}, i=1,\dots, g-1$ to form a nodal curve $C$ of genus $g$.

On a reducible curve, stability for a bundle depends on a choice of polarization.
A polarization is usually defined as the choice of a line bundle on the variety.
For our goal of defining stability of a vector bundle, what matters is the relative degree of the restriction of this line bundle to each component, that is the numbers 
\[ a_i, i=1,\dots, g, 0<a_i<1,\ \sum a_i=1\]
where each $a_i$ can be interpreted as $a_i=\frac{\deg L_{|C_i}}{\deg L}$.
Then, a vector bundle $E$ on $C$ of constant rank $r$ is said to be $a_i$-stable if for every subsheaf $F$ of $E$
\[ \frac {\chi (F)}{\sum a_i\text{rank} (F_{|C_i})}<\frac{\chi(E)}r  . \]
Fixing a polarization, there is a moduli space of vector bundles on  the chain of elliptic curves (see \cite{S}).
The polarization forces the degree of the restriction on each component to  vary in certain intervals.
In fact, the moduli space is reducible, each component corresponding to a choice of allowable degrees on the components (see \cite{modtree}, \cite{modredcur}).
Our goal is to use results from  the chain of elliptic curves to deduce similar conditions for  non-singular curves. 
When dealing with  a family of curves in which the general member is generic and the special  member is the chain of elliptic curves,
we can modify a vector bundle in the family tensoring with a line bundle with support on the special fiber.
This action leaves the vector bundle on the general fiber unchanged but moves the degree on the various components of the special fiber by multiples of the rank. 
This allows us  to choose the distribution of degrees among the components up to multiples of $r$ 
and  ignore the restriction imposed by the polarization on degrees of restriction on components.
We can focus then on the remaining conditions needed for stability.

We will show the following

\begin{proposition}\label{propredcur}
Let $C$ be a chain of elliptic curves of genus $g$.
 Fix a rank $r$ and degree $d, 0\le d<r$ and a collection of integers $r_1<r_2<\dots<r_k=r$. 
 Choose  degrees $d_1,\dots, d_k=d$ with $d_{k-1}$  the largest degree such that $\frac {d_{k-1}}{r_{k-1}}<\frac {d_k}{r_k}$,
  $d_{k-2}$  the largest degree such that $\frac {d_{k-2}}{r_{k-2}}<\frac {d_{k-1}}{r_{k-1}}$,\dots, 
   $d_{1}$  the largest degree such that $\frac {d_{1}}{r_{1}}<\frac {d_2}{r_2}$.
 Then,  there exists a stable bundle $E$  that contain a chain of subbundles 
 \[ E_{r_1,d_1}\subseteq  E_{r_2,d_2}\subseteq \dots \subseteq  E_{r_k,d_k}=E\]
  with $E_{r_i,d_i}$ stable of degree $d_i$ and rank $r_i$.. Moreover, this $E$ contains at most a finite number of such chains.
 \end{proposition}
 \begin{proof}  On the moduli space of vector bundles on the chain of elliptic curves, we focus on  the component whose restriction 
 have degree $d<r$ on $C_1$ and degree zero on the remaining components.
Write $h$ for the greatest common divisor of $d, r$. Define $d', r'$ by  $d=hd', r=hr'$.
On the chosen component of the moduli space of vector bundles, the generic vector bundle restricts to a direct sum of $h$ indecomposable bundles of degree $d'$ and rank $r'$ on $C_1$
and as a direct sum of line bundles of degree 0 on $C_2, \dots, C_g$.

More generally, for any $r_i, d_i$, we will say that  $E_{r_i,d_i}$ is a generic vector bundle of degree $d_i$ and rank $r_i$ if it is a direct sum of $h_i$ indecomposable vector bundles of 
coprime rank and degree
\[ E_{r_i,d_i}=\oplus _{j=1}^{h_i} F^j_i, \  \ h_i=gcd(r_i, d_i), \ r_i=h_ir'_i, \ d_i=h_id'_i, \ \deg F_i=r'_i, \ rkF_i=r'_i.\]

On $C_1$, choose  a generic vector bundle $E^1_{r_1,d_1}$ of degree $d_1$ and rank $r_1$,  a generic vector bundle $E^1_{r_2,d_2}$ of degree $d_2$ and rank $r_2$, 
\dots,  a generic vector bundle $E^1_{r_k,d_k}$ of degree $ d_k$ and rank $r_k$.

The conditions  $ \frac {d_1}{r_1}< \frac {d_2}{r_2}<\dots<\frac {d_k}{r_k}$ guarantee (see \cite{langred}, Lemma 2.5) that there exist inclusions 
\[ E^1_{r_1,d_1}\subseteq  E^1_{r_2,d_2}\subseteq \dots \subseteq  E^1_{r_k,d_k}\]
In fact, as $E^1_{r_i,d_i}=\oplus _{j=1}^{h_i} F^j_i,  E^1_{r,d}=E_{r_k,d_k}=\oplus _{j=1}^{h} F^j_k$, 
$Hom (E^1_{r_i,d_i}, E^1)=\oplus _{j, j'}Hom (F^j_i,F^{j'}_k)$.
Then,   from 
\[ \frac {d'_{i}}{r'_{i}}=\frac {d_{i}}{r_{i}}<\frac {d}{r}=\frac {d'}{r'},\]
  the space of morphisms of $F^j_i$ to $ F^{j'}_k$ has dimension $r'_id'-r'd_i'$.
Therefore,  the space of morphisms of $E^1_{r_i,d_i}$ to $ E^1_{r,d}$ has dimension $hh_i(r'_id'-r'd_i')=r_id-rd_i\ge 0$,  so it is  non-empty.
We can choose the inclusions from $E^1_{r_i,d_i}$ into $E_{r_k,d_k} $ so that the image does not coincide with any of the finite number of destabilizing subsheaves of $E_{r_k,d_k}$
(it is enough to make sure that none of the morphisms $F^j_i$ to $ F^{j'}_k$ is zero).

We now describe a vector bundle on the chain by giving a vector bundle on each component and the gluing at the nodes:

On the curve $C^1$ take the vector bundle $E^1_{r_k,d_k}=E^1_{r,d}$ we just described.
On the curves $C_2,\dots, C_g$, choose direct sum of $r$ line bundles of degree 0.
On each of $C_2,\dots, C_g$, select a first set of $r_1$ among the $r$ line bundles in the direct sum.
 Select then a second set of $r_2$ among the $r$ line bundles containing the initial subset of $r_1$ already chosen,
 Select  a third set of $r_3$  line bundles containing the  subset of $r_2$ chosen in the previous step and so on.
 Form now a bundle on the chain of elliptic curves by gluing the bundles on each component so that 
 when identifying $Q_i$ with $P_{i+1}, i=2,\dots, g-1$ each of the sets of  $r_j$ line bundles $j=1,\dots, k$ on $C_i$ glues with the set of  $r_j$ line bundles on $C_{i+1}$, $j=1,\dots, k$
  (but the gluings are otherwise generic).  
At $Q_1$, glue each set of  $r_j$ line subbundles on $C_2$   with the fiber of the image of the  $E^1_{r_j,d_j}$ (but the gluings are otherwise generic).
 In this way, we obtain bundles  of ranks $r_1<r_2<\dots<r_k$  and degrees $d_1,\dots, d_k$ on the whole curve $C$ each contained in the next.
 
On a reducible nodal curve, gluing  vector bundles that are semistable  on each of the components and of the degrees allowed by the polarization,
 one obtains a semistable bundle on the whole curve.
Moreover, if one of the bundles we are gluing is stable or if none of the subbundles that contradict stability glue with each other, the whole vector bundle on the reducible curve is stable
(see \cite{modredcur}, \cite{modtree}).

By construction, the vector bundles on each $C_i$ are semistable.
On $C_1$, the only subbundles of the bundle $E^1_{r_k,d_k}=\oplus _{j=1}^hE^j_{r',d'}=\oplus _{j=1}^hF^j_k$  that contradict stability 
are the $h$ subsheaves $F^j_k$  of degree $d'$ and rank $r'$ and their direct sums.
Our choice of the inclusions of the subbundles in the bundle on $C_1$ and the gluings at the nodes guarantee that we have a stable overall bundle. 

Note also that our choice of $d_i$ means that  $r_id_{i+1}-r_{i+1}(d_i+1)\le -1$ or equivalently $r_id_{i+1}-r_{i+1}d_i\le r_{i+1}-1$ 
In the interval, $1\le r_i\le r_{i+1}-1$, this implies that $r_id_{i+1}-r_{i+1}d_i\le r_i( r_{i+1}-r_i)$.
Therefore, given a subspace of dimension $r_i$ of the fiber of  $E^1_{r_{i+1},d_{i+1}}$ at $Q_1$, there is at most a finite number of subbundles $E^1_{r_i,d_i}$ whose immersion in 
$E^1_{r_{i+1},d_{i+1}}$ glue with that fixed subspace (see Proposition 2.8 of \cite{langred}).
Therefore, the number of chains for a fixed $E_{r,d}$ on the reducible curve is finite.
 \end{proof}


\section{Extending the result to the non-singular curve.}

We start by using the results on the reducible curve to extend it to a generic, non-singular curve
\begin{proposition}\label{propgencur}
Let $C$ be a generic curve of genus $g$.
 Fix a rank $r$ and degree $d, 0\le d<r$ and a collection of integers $r_1<r_2<\dots<r_k=r$. 
 Choose  degrees $d_1,\dots, d_k=d$ with $d_{k-1}$  the largest degree such that $\frac {d_{k-1}}{r_{k-1}}<\frac {d_k}{r_k}$,
  $d_{k-2}$ is the largest degree such that $\frac {d_{k-2}}{r_{k-2}}<\frac {d_{k-1}}{r_{k-1}}$,\dots 
   $d_{1}$  the largest degree such that $\frac {d_{1}}{r_{1}}<\frac {d_2}{r_2}$.
 Then,  there exists a stable bundle $E$  that contain a chain of subbundles 
 \[ E_{r_1,d_1}\subseteq  E_{r_2,d_2}\subseteq \dots \subseteq  E_{r_k,d_k}=E\]
  with $E_{r_i,d_i}$ stable of degree $d_i$ and rank $r_i$.
 \end{proposition}
 \begin{proof} Take a family of curves where the special fiber is a chain of elliptic curves and the generic curve is non-singular.
 Then, the result follows from Proposition \ref{propredcur}  using the openness of the stability condition.
 \end{proof}.
 
 We proved stability of the various steps of a chain of extensions under the harder conditions on slopes. 
 This implies the similar result when the slopes are not as close:
 \begin{proposition}\label{propgencuranydeg}
Let $C$ be a generic curve of genus $g$.
 Fix a rank $r$ and degree $d, 0\le d<r$ and two  collections of integers $r_1<r_2<\dots<r_k=r$,  $d_1,\dots, d_k=d$ 
such that  
 \[ \frac {d_1}{r_1}< \frac {d_2}{r_2}<\dots<\frac {d_k}{r_k}\]
 Then,  there exists a stable bundle $E$  that contain a chain of subbundles 
 \[ E_{r_1,d_1}\subseteq  E_{r_2,d_2}\subseteq \dots \subseteq  E_{r_k,d_k}=E\]
  with $E_{r_i,d_i}$ stable of degree $d_i$ and rank $r_i$.
 \end{proposition}
 \begin{proof} Fix integers $r, d, r_1, d_1$ with $ \frac {d_1}{r_1}< \frac {d}{r}$.
  The set of vector bundles of rank $r$ and degree $d$ which contain a subbundle of rank $ r_1$ and degree $ d_1-1$ is contained in the closure of those vector bundles that contain a 
  subbundle of rank $ r_1$ and degree $ d_1$ (see \cite{RT} Corollary 1.12)
  Therefore, the result follows from \ref{propgencur}.
  \end{proof}.
 
Let us now look at dimension and irreducibility:

\begin{proposition} \label{propIrrFamExt} Fix integers $d_1, d_2, r_1, r_2$ such that $\frac {d_1}{r_1}<\frac {d_2}{r_2}$.
Let ${\mathcal U}_1$ be an irreducible family of stable vector bundles of rank $r_1$ and degree $d_1$.
Let $\bar{{\mathcal U}_2}$ be an irreducible family of stable vector bundles of rank $r_2-r_1$ and degree $d_2-d_1$.
Then, the family of   extensions 
\[ 0\to E_1\to E\to \bar{E_2}\to 0, \  E_1\in {\mathcal U}_1, \bar{E_2}\in \bar{{\mathcal U}_2} \]
is also irreducible of dimension 
\[ \dim ( {\mathcal U}_1)+\dim ( \bar{{\mathcal U}_2})+r_1(r_2-r_1)(g-1)+r_1d_2-r_2d_1-1 \]. 
\end{proposition}
\begin{proof} For fixed $E_1,  \bar{E_2}$, the space of extensions as above is parameterized by $H^1( \bar{E_2}^*\otimes E_1)$.
We claim that  $H^0( \bar{E_2}^*\otimes E_1)= 0$. If this were not the case then there would be a non-zero  morphism of $ \bar{E_2}\to E_1$.
Its image $I$ would be  both a quotient of $ \bar{E_2}$ and a subsheaf of $ E_1$.
The stability of the two bundles implies that 
\[ \frac {d_1}{r_1}=\mu(E_1)<\mu (I)<\mu(\bar{E_2})=\frac {d_2-d_1}{r_2-r_1}\]
This contradicts the assumption of our initial choice of ranks and degrees.
It follows that $H^0( \bar{E_2}^*\otimes E_1)= 0$ and therefore $H^1( \bar{E_2}^*\otimes E_1)$ has dimension $r_1d_2-r_2d_1+r_1(r_2-r_1)(g-1)$, irrespectively of the choice of $E_1, \bar{E_2}$.
Then the statement about the dimension follows.
\end{proof}

\begin{proof}(of Theorem \ref{mainth}).
Take ${\mathcal U}_1$  the space of all vector bundles of degree $d_1$ and rank $r_1$, 
$\bar{\mathcal U}_2$  the space of all vector bundles of degree $d_2-d_1$ and rank $r_2-r_1$, \dots 
$\bar{\mathcal U}_k$  the space of all vector bundles of degree $d_k-d_{k-1}$ and rank $r_k-r_{k-1}$.
From  Proposition \ref{propgencur}, the set of {\bf stable} bundles $E$ that can be obtained by successive extensions is non-empty.
Proposition \ref{propIrrFamExt} allows us to compute successively the dimensions of the space of extensions,
starting with  
\[ \dim ({\mathcal U}_1)=r_1^2(g-1)+1, \dim (\bar{\mathcal U}_2)=(r_2-r_1)^2(g-1)+1,\dots,  \dim (\bar{\mathcal U}_k)=(r_k-r_{k-1})^2(g-1)+1\]
The last claim in Proposition \ref{propredcur} ensures that each vector bundle appears only a finite number of times as an extension of the given form.
\end{proof}

\bigskip

\


\bigskip

\bigskip


 
 
 \begin{thebibliography}{10}
 \bibitem[BR]{BR} E. Ballico, B. Russo,
{\sl Exact sequences of semistable vector bundles on algebraic curves},
Bull. London Math. Soc. 32 (2000), no. 5, 537-546
 
 \bibitem[BL]{BL} 
L. Brambila-Paz, L, H.Lange, Herbert {\sl A stratification of the moduli space of vector bundles on curves.}
J. Reine Angew. Math. {\bf 494} (1998), 173–187.
 
 \bibitem[L]{L}
H.Lange, {\sl Zur Klassifikation von Regelmannigfaltigkeiten}, Math.Ann. \textbf{262} (1983), p. 447-459. 

 \bibitem[LN]{LN}  H.Lange, M.S.Narasimhan, {\sl Maximal subbundles of rank two vector bundles on curves}, Math.Ann. {\bf 266} (1983), 55-72.



\bibitem[RT]{RT}
B.Russo, M.Teixidor, {\sl On a conjecture of Lange},  J.Alg.Geom. \textbf{8} (1999), p. 483-496.

\bibitem[S]{S}
C.S. Seshadri , {\sl  Fibr\'es vectoriels sur les courbes alg\'ebriques. } Ast\'erisque, 96. Soci\'et\'e Math\'ematique de France, Paris, 1982. 209 pp

\bibitem[T1]{modtree}
 M.Teixidor, {\sl Moduli spaces of (semi)stable vector bundles on tree-like curves. },  Math. Ann. 290 (1991), no. 2, 341-348.

\bibitem[T2]{modredcur}
 M.Teixidor, {\sl Moduli spaces of vector bundles on reducible curves},  
 Amer. J. Math. 117 (1995), no. 1, 125-139.


\bibitem[T3]{extlb}
 M.Teixidor, {\sl Stable extensions by line bundles },   J. Reine Angew. Math. 502 (1998), 163-173.

\bibitem[T4]{langred}
 M.Teixidor, {\sl On Lange's conjecture},   J. Reine Angew. Math. 528 (2000), 81-99.

\end{thebibliography}
\end{document}